\documentclass{amsart}[12pt]
\usepackage{amssymb,amsmath}
\usepackage{enumerate}
\usepackage[english]{babel}
\usepackage{color}
\usepackage{comment}

\newtheorem{teo}{Theorem}[section]
\newtheorem{pro}{Proposition}[section]
\newtheorem{cor}{Corollary}[section]

\newtheorem{lm}{Lemma}[section]

\theoremstyle{definition}
\newtheorem{rem}{Remark}[section]

\newtheorem{ex}{Example}[section]

\title[Noncommutative modulated ergodic theorems]{A noncommutative weak type maximal inequality for modulated ergodic averages with general weights}
\keywords{Semifinite von Neumann algebra, noncommutative weighted individual ergodic theorems, multiparameter individual ergodic theorems, Besicovitch sequences, Hartman sequences}
\subjclass[2020]{47A35, 46L51}

\author{Morgan O'Brien}
\address{North Dakota State University\\ Department of Mathematics\\ 1210 Albrecht Boulevard, Minard Hall \\ Fargo, ND 58102, USA}
\email{morgan.obrien@ndsu.edu, obrienmorganc@gmail.com}

\begin{document}
\begin{abstract}
In this article, we prove a weak type $(p,p)$ maximal inequality, $1<p<\infty$, for weighted averages of a positive Dunford-Schwarz operator $T$ acting on a noncommutative $L_p$-space associated to a semifinite von Neumann algebra $\mathcal{M}$, with weights in $W_q$, where $\frac{1}{p}+\frac{1}{q}=1$. 
This result is then utilized to obtain modulated individual ergodic theorems with $q$-Besicovitch and $q$-Hartman sequences as weights. Multiparameter versions of these results are also investigated.
\end{abstract}
\date{August 8, 2023}
\maketitle

\section{Introduction}\label{s1}

Since Ryll-Nardzewski proved that bounded Besicovitch sequences are good weights for the individual ergodic theorem for Dunford-Schwartz operators \cite{rn}, the study of modulated ergodic theorems has been an active area of research in ergodic theory. For example, such results have been studied by Bellow and Losert for certain bounded Hartman sequences with correlation \cite{bl} (this class contains all bounded Besicovitch sequences), $q$-Besicovitch sequences by Lin, Olsen, and Tempelman \cite{lot}, and certain sequences arising from arithmetic functions in number theory by El Abdalaoui, Kulaga-Przymus, Lema\'{n}czyk, and de la Rue \cite{eakpldlr} and Cuny and Weber \cite{cw}, just to name a few. Also, \c{C}\"{o}mez and Litvinov generalized bounded Besicovitch sequences to the setting of unimodular functions and obtained individual ergodic theorems along such sequences; furthermore, they also generalized these results to the case of certain super-additive processes \cite{comli2}. Multiparameter versions of some of these results are known as well; for example, Jones and Olsen \cite{jo} showed that multi-parameter $q$-Besicovitch sequences satisfy these types of results when the operators considered are positive contractions of a fixed $L_p$-space.

Generalizing results of this nature from the classical measure space setting to the von Neumann algebra setting is an active area of research in noncommutative ergodic theory. 
For example, Chilin, Litvinov, and Skalski \cite{cls} showed that bounded Besicovitch sequences are good weights for the noncommutative individual ergodic theorem (see also \cite{cl1,cl2}). \c{C}\"{o}mez and Litvinov have also proved similar results for operator-valued Besicovitch sequences (see \cite{comli3}). Related to these, Litvinov \cite{li2} showed that a noncommutative version of the Wiener-Wintner ergodic theorem holds for ergodic $\tau$-preserving $*$-homomorphisms of a finite von Neumann algebra with weights being trigonometric polynomials. Hong and Sun \cite{hs} showed that the weights considered by Bellow and Losert also satisfy a Wiener-Wintner type result in a multiparameter form for certain $\tau$-preserving $*$-automorphisms of finite von Neumann algebras. In \cite{ob2}, under some strong assumptions on the positive Dunford-Schwartz operator under consideration (frequently satisfied for operators satisfying certain functional analytic properties, like being self-adjoint on $L_2$), the author was able to prove that all bounded Hartman sequences satisfy a Wiener-Wintner type result.

In the noncommutative setting, the maximal inequalities used only allow bounded weights to be used, even though stronger results are true in the commutative setting with unbounded weights (see \cite{lot,jo}). However, the assumptions made on the operators in \cite{ob2} actually allow unbounded weights to be used in some special cases in the von Neumann algebra setting. The methods used there are somewhat specialized, however, and do not lead to an immediate path to generalization for all positive Dunford-Schwartz operators.

In this article, we will prove a maximal inequality suitable to proving that weights in $W_q$ (see Section \ref{s22} for a definition) may be used in modulated ergodic theorems for any positive Dunford-Schwartz operators acting on von Neumann algebras (see Theorem \ref{t31}). The version we prove actually applies in a multiparameter setting with some extra restrictions that are not needed in the commutative setting (see Remark \ref{r31}). However, even in the single parameter case our results seem to be new in the von Neumann algebra setting. Afterwards, we use this result to prove some extensions of some of the results in \cite{cl1,cls,li2,ob2} to incorporate these more general class of weights. In particular, in Theorem \ref{t32} we extend the results of \cite{cl1,cl2} to allow $q$-Besicovitch sequences as weights, in Theorem \ref{t33} we expand on the special cases of \cite{ob2} to allow $q$-Hartman sequences in other situations, and in Remark \ref{r32} we discuss how to extend the Wiener-Wintner result of \cite{li2} to allow $q$-Besicovitch sequences.

\section{Preliminaries and Notation}\label{s2}

We will let $\mathbb{N}$ denote the set of natural numbers and $\mathbb{N}_0=\mathbb{N}\cup\{0\}$. Write $\mathbb{T}:=\{z\in\mathbb{C}:|z|=1\}$ to denote the unit circle. Throughout this section, let $d\in\mathbb{N}$ be fixed. Also, let $\textbf{0}:=(0,...,0)_{\textbf{n}\in\mathbb{N}_0^d}\subset\mathbb{C}^d$ denote the constant $0$ sequence.

Throughout this paper we will adopt the convention of
$$\frac{0}{0}:=0.$$ 
This convention is needed in Theorem \ref{t31}, where for $\epsilon>0$ we wish to obtain a quantity of the form $\sup_{\alpha}\frac{1}{\alpha}M_\alpha\leq\epsilon$, where $\alpha$ may be $0$. Hence, one does not need to do a separate case for $\alpha=0$; $M_\alpha$ is the quantity we are mostly concerned with, and it will be the case that $M_\alpha=0$ when $\alpha=0$, so that $\frac{1}{\alpha}M_\alpha=\frac{0}{0}=0<\epsilon$ with this convention.


Given a sequence $(\textbf{n}_k)_{k=0}^{\infty}\subseteq\mathbb{N}_0^d$, we will say that $\textbf{n}_k\to\infty$ if $n_j(k)\to\infty$ for every $j\in\{1,...,d\}$, where $\textbf{n}_k=(n_1(k),...,n_d(k))$ for every $k\in\mathbb{N}_0$.

Let $(x_{\textbf{n}})_{\textbf{n}\in\mathbb{N}_0^d}$ be a sequence in a Banach space $\mathcal{X}$. Then we will say that $x_{\textbf{n}}\to x$ as $\textbf{n}\to\infty$ if for every $\epsilon>0$ there exists $N\in\mathbb{N}_0$ such that, for  $\textbf{n}=(n_1,...,n_d)\in\mathbb{N}_0^d$, 
$$\min\{n_1,...,n_d\}\geq N \ \text{ implies } \ \|x_{\textbf{n}}-x\|_{\mathcal{X}}<\epsilon.$$

Similarly, if $(\alpha_{\textbf{n}})_{\textbf{n}\in\mathbb{N}_0^d}\subset\mathbb{R}$, then we will write
$$\limsup_{\textbf{n}\to\infty}\alpha_{\textbf{n}}=\lim_{N\to\infty}\sup_{\min\{n_1,...,n_d\}\geq N}\alpha_{n_1,...,n_d}.$$

If $\textbf{n}=(n_1,...,n_d)\in\mathbb{N}^d$, then we will use the notation 
$$\frac{1}{|\textbf{n}|}\sum_{\textbf{k}=0}^{\textbf{n}-1}:=\frac{1}{n_1\cdots n_d}\sum_{k_1=0}^{n_1-1}\cdots\sum_{k_d=0}^{n_d-1}.$$

Given $1\leq C<\infty$, let $$\textbf{N}_C^{(d)}=\left\{(n_1,...,n_d)\in\mathbb{N}^d:\frac{n_i}{n_j}\leq C\text{ for every }1\leq i,j\leq d\text{ with }n_j\neq0\right\}.$$ A sequence $(\textbf{n}_k)_{k=0}^{\infty}\subseteq\mathbb{N}^d$ \textit{remains in a sector of $\mathbb{N}^d$} if there exists $1\leq C<\infty$ such that $\textbf{n}_{k}\in \textbf{N}_C^{(d)}$ for every $k\in\mathbb{N}_0$. When $d=1$, without loss of generality we may assume that $C=1$ and $\textbf{N}_1^{(1)}=\mathbb{N}$.

\subsection{Noncommutative $L_p$-spaces}

If $\mathcal{M}$ is a von Neumann algebra and $\tau$ is a normal semifinite faithful trace on $\mathcal{M}$, then we will call the pair $(\mathcal{M},\tau)$ a semifinite von Neumann algebra. Let $\textbf{1}$ denote the identity operator of $\mathcal{M}$. Let $\mathcal{P}(\mathcal{M})$ denote the set of projections in $\mathcal{M}$, and for each $e\in\mathcal{P}(\mathcal{M})$ write $e^\perp:=\textbf{1}-e$.

Suppose that $\mathcal{M}$ acts on the Hilbert space $\mathcal{H}$. Let $x:\mathcal{D}_x\to\mathcal{H}$ be a closed densely defined operator on $\mathcal{H}$. Then $x$ is \textit{affiliated} to $\mathcal{M}$ if $yx\subseteq xy$ for every $y\in\mathcal{M}'$ ($\mathcal{M}'\subseteq\mathcal{B}(\mathcal{H})$ being the commutant of $\mathcal{M}$). If $x$ is affiliated to $\mathcal{M}$, then it is called \textit{$\tau$-measurable} if for every $\epsilon>0$, there exists $e\in\mathcal{P}(\mathcal{M})$ such that $\tau(e^\perp)\leq\epsilon$ and $xe\in\mathcal{M}$. Let $L_0(\mathcal{M},\tau)$ denote the set of all $\tau$-measurable operators affiliated with $\mathcal{M}$. For every $\epsilon,\delta>0$, define the set
$$V(\epsilon,\delta)=\{x\in L_0(\mathcal{M},\tau):\|xe\|_\infty\leq\delta\text{ for some }e\in\mathcal{P}(\mathcal{M})\text{ with }\tau(e^\perp)\leq\epsilon\}.$$
The collection $\{V(\epsilon,\delta)\}_{\epsilon,\delta>0}$ forms a set of neighborhoods of $0$ in $L_0(\mathcal{M},\tau)$, which will give rise to the \textit{measure topology} of $L_0(\mathcal{M},\tau)$. Equipped with the closed sum, closed product, and measure topology, $L_0(\mathcal{M},\tau)$ is a complete topological $*$-algebra. See \cite{ne} for more information on this.

If $x\in L_0(\mathcal{M},\tau)$, then $x$ is \textit{positive}, written $x\geq0$, if $\langle x\xi,\xi\rangle\geq0$ for every $\xi\in\mathcal{D}_x$. Write $x\leq y$ if $y-x\geq0$, where $x,y\in L_0(\mathcal{M},\tau)$ and $x,y\geq0$. If $E\subseteq L_0(\mathcal{M},\tau)$, write $E^+=\{x\in E:x\geq0\}$. A linear operator $S:E\to E$ is called \textit{positive} if $S(x)\geq0$ for every $x\geq0$ (i.e. $S(E^+)\subseteq E^+$).

Given $x\in L_0(\mathcal{M},\tau)^+$, by the spectral theorem one may write $x$ in its spectral decomposition as $x=\int_{[0,\infty)}\lambda de_\lambda$. From this, one may extend the trace from $\mathcal{M}^+$ to $L_0(\mathcal{M},\tau)^+$ via
$$\tau(x)=\sup_{n}\tau\left(\int_{[0,n]}\lambda de_\lambda\right).$$

If $x\in L_0(\mathcal{M},\tau)$, then there exists $u\in\mathcal{M}$ and $|x|\in L_0(\mathcal{M},\tau)^+$ such that $x$ has a polar decomposition $x=u|x|$, where $|x|^2=x^*x$. For each $1\leq p<\infty$, let $\|x\|_p:=\tau(|x|^p)^{1/p}$ for $x\in L_0(\mathcal{M},\tau)$, and let the noncommutative $L_p$-space associated to $\mathcal{M}$ be defined by
$$L_p(\mathcal{M},\tau)=\{x\in L_0(\mathcal{M},\tau):\|x\|_p<\infty\}.$$ Write $L_\infty(\mathcal{M},\tau)=\mathcal{M}$, and equip it with the usual operator norm $\|\cdot\|_\infty$. Then $L_p(\mathcal{M},\tau)$ is a Banach space with respect to $\|\cdot\|_p$ for every $1\leq p\leq\infty$. See \cite{ne} for more on this. If $p=0$ or $1\leq p\leq\infty$, we may write $L_p=L_p(\mathcal{M},\tau)$ when convenient and unambiguous. It is known that $L_p\subset L_1+\mathcal{M}$ for every $1\leq p\leq\infty$.


A linear operator $T:L_1+\mathcal{M}\to L_1+\mathcal{M}$ is a \textit{Dunford-Schwartz operator} if
$$\|T(x)\|_p\leq\|x\|_p\text{ for every }x\in L_p\text{ and }1\leq p\leq\infty.$$
Let $DS^+(\mathcal{M},\tau)$ denote the set of all positive Dunford-Schwartz operators on $(\mathcal{M},\tau)$. Given $\textbf{T}=(T_1,...,T_d)\in DS^+(\mathcal{M},\tau)^d$ and $\textbf{k}=(k_1,...,k_d)\in\mathbb{N}_0^d$, we will write $\textbf{T}^{\textbf{k}}=T_1^{k_1}\cdots T_d^{k_d}$.

There are a few properties of the ordering on $\mathcal{M}^+$ that are vital to our arguments. If $f:[0,\infty)\to[0,\infty)$ is an operator convex function (i.e. $f$ is still convex if we replace $[0,\infty)$ with $\mathcal{M}^+$ using functional calculus), then for every positive operator $S:\mathcal{M}\to\mathcal{M}$ it follows that $f(S(x))\leq S(f(x))$ for every $x\in\mathcal{M}^+$. Important to us is that $f(t)=t^p$ is operator convex when $1\leq p\leq2$; Kadison's inequality is the case $p=2$. In a similar vein to this, the function $f$ is said to be operator monotone if $x,y\in\mathcal{M}^+$ satisfying $x\leq y$ implies $f(x)\leq f(y)$. The function $f(t)=t^{\frac{1}{p}}$ is operator monotone for every $1\leq p<\infty$. For more on this, see \cite{da,jy}.

The map $0\leq t\mapsto t^p$ cannot necessarily be improved for operator convexity for $p>2$. For example, the case $p=3$ fails when considering the matrices
$$x=\left[\begin{array}{cc}1&1\\1&1\end{array}\right] \, \text{ and } \, y=\left[\begin{array}{cc}3&1\\1&1\end{array}\right].$$

Using the notion of operator convexity, the following operator H\"{o}lder's inequality was shown in \cite[Lemma 2.4]{hlrx}. There, the result was proven for arbitrary measure spaces. However, we will only need it for averages of a finite number of terms, so we state it in that language.

\begin{lm}(Cf. \cite[Lemma 2.4]{hlrx})\label{l21}
Let $(\mathcal{M},\tau)$ be a semifinite von Neumann algebra and assume $1<p,q<\infty$ satisfy $\frac{1}{p}+\frac{1}{q}=1$. Fix $n\in\mathbb{N}$, and let $\alpha_0,...,\alpha_{n-1}\in[0,\infty)$ and $x_0,...,x_{n-1}\in(L_1(\mathcal{M},\tau)+\mathcal{M})^+$. Then
$$\frac{1}{n}\sum_{k=0}^{n-1}\alpha_kx_k
\leq\left(\frac{1}{n}\sum_{k=0}^{n-1}\alpha_k^q\right)^{\frac{1}{q}}\left(\frac{1}{n}\sum_{k=0}^{n-1}x_k^p\right)^{\frac{1}{p}}.$$
\end{lm}

\subsection{$W_q^{(d)}$-spaces and weights}\label{s22}
If $1\leq q<\infty$ and $\alpha=(\alpha_\textbf{n})_{\textbf{n}\in\mathbb{N}_0^d}\subset\mathbb{C}$, then 
$$\|\alpha\|_{W_q^{(d)}}:=\left(\limsup_{\textbf{n}\to\infty}\frac{1}{|\textbf{n}|}\sum_{\textbf{k}=0}^{\textbf{n}-1}|\alpha_{\textbf{k}}|^q\right)^{\frac{1}{q}}$$ will denote the $W_q^{(d)}$-seminorm of $\alpha$. Given $C\geq1$, we will write
$$|\alpha|_{W_{q,C}^{(d)}}:=\left(\sup_{\textbf{n}\in \textbf{N}_C^{(d)}}\frac{1}{|\textbf{n}|}\sum_{\textbf{k}=0}^{\textbf{n}-1}|\alpha_{\textbf{k}}|^q\right)^{\frac{1}{q}}.$$
\begin{lm}
With the notation above, if $d=1$, then $\|\alpha\|_{W_q^{(1)}}<\infty$ if and only if  $|\alpha|_{W_{q,1}^{(1)}}<\infty$. If $d\geq2$, then $\|\alpha\|_{W_q^{(d)}}<\infty$ implies $|\alpha|_{W_{q,C}^{(d)}}<\infty$ for every $C\geq1$.
\end{lm}
\begin{proof}
The case $d=1$ is straightforward. Hence, assume $d\geq2$ and $C\geq1$. Let $N\in\mathbb{N}$ be such that $\frac{1}{|\textbf{n}|}\sum_{\textbf{k}=0}^{\textbf{n}-1}|\alpha_{\textbf{k}}|^q<\|\alpha\|_{W_q^{(d)}}^q+1$ for every $\textbf{n}=(n_1,...,n_d)\in\mathbb{N}^d$ with $\min\{n_1,...,n_d\}\geq N$.

Let $A_N=\{(m_1,...,m_d)\in\textbf{N}_C^{(d)}:\min\{m_1,...,m_d\}< N\}$. If $(n_1,...,n_d)\in A_N$, then there exists $i$ such that $n_i< N$. Then, for every $j\in\{1,...,d\}$, by definition of being in $\textbf{N}_C^{(d)}$ we find that $n_j\leq Cn_i<CN$. Since $j$ was arbitrary, it follows that $(n_1,...,n_d)\in([0,CN]\cap\mathbb{N})^d$. Therefore it follows that $A_N\subset([0,CN]\cap\mathbb{N})^d$, which means $A_N$ is a finite set, and so
$$\left(\sup_{\textbf{n}\in A_N}\frac{1}{|\textbf{n}|}\sum_{\textbf{k}=0}^{\textbf{n}-1}|\alpha_{\textbf{k}}|^q\right)^{1/q}<\infty.$$
Combining this with how $N$ was chosen shows that the supremum over all of $\textbf{N}_C^{(d)}$ is finite as well.
\end{proof}

We define $W_q^{(d)}$ as
$$W_q^{(d)}=\left\{\alpha=(\alpha_\textbf{n})_{\textbf{n}\in\mathbb{N}_0^d}\subset\mathbb{C}:\|\alpha\|_{W_q^{(d)}}<\infty\right\}.$$ Let $W_\infty^{(d)}$ denote the space of all $\alpha=(\alpha_{\textbf{n}})_{\textbf{n}\in\mathbb{N}_0^d}\subset\mathbb{C}$, with $\sup_{\textbf{n}}|\alpha_\textbf{n}|<\infty$, and let $|\alpha|_{W_{\infty,C}^{(d)}}=\|\alpha\|_{W_\infty^{(d)}}=\sup_{\textbf{n}\in\mathbb{N}_0^d}|\alpha_{\textbf{n}}|$. If $1\leq r<s\leq\infty$, then $W_s^{(d)}\subset W_r^{(d)}$, and $\|\alpha\|_{W_r^{(d)}}\leq\|\alpha\|_{W_s^{(d)}}$ and $|\alpha|_{W_{r,C}^{(d)}}\leq|\alpha|_{W_{s,C}^{(d)}}$ for every $\alpha$. For every $1\leq q\leq\infty$ it is known that $\|\cdot\|_{W_q^{(d)}}$ defines a seminorm on $W_q^{(d)}$ under which it is complete.

Given $\mathcal{W}\subseteq W_1^{(d)}$, we will write $\mathcal{W}^+$ for the set of sequences $\alpha=(\alpha_\textbf{n})_{\textbf{n}\in\mathbb{N}_0^d}\in\mathcal{W}$ such that $\alpha_\textbf{n}\geq0$ for every $\textbf{n}\in\mathbb{N}_0^d$. Note that any $\alpha\in W_q^{(d)}$ can be written as a linear combination of four elements of $\alpha_0,...,\alpha_3\in (W_q^{(d)})^+$ for each $1\leq q\leq\infty$, and these elements can be chosen so that $|\alpha_j|_{W_{q,C}^{(d)}}\leq|\alpha|_{W_{q,C}^{(d)}}$ for each $j=0,...,3$.

For any $\underline{\lambda}=(\lambda_1,...,\lambda_d)\in\mathbb{T}^d$ and $\textbf{k}=(k_1,...,k_d)\in\mathbb{Z}^d$, we will write $\underline{\lambda}^{\textbf{k}}=\lambda_1^{k_1}\cdots \lambda_d^{k_d}$. A trigonometric polynomial is a function $P:\mathbb{Z}^d\to\mathbb{C}$ such that there exists $\underline{\lambda}_1,...,\underline{\lambda}_k\in\mathbb{T}^d$ and $r_1,...,r_k\in\mathbb{C}$ such that $P(\textbf{n})=\sum_{j=1}^{k}r_j\underline{\lambda}_j^{\textbf{n}}$ for every $\textbf{n}\in\mathbb{Z}^d$. A trigonometric polynomial $P$ will also induce a sequence in $W_\infty^{(d)}$ via $(P(\textbf{n}))_{\textbf{n}\in\mathbb{N}_0^d}$, and denote $\mathcal{T}^{(d)}\subset W_\infty^{(d)}$ for the subspace of all such sequences.

There are a few important subspaces of $W_1$ that we wish to mention in particular. To begin with, we will write $W_{1^+}^{(d)}$ for the closure of $\bigcup_{r>1}W_r^{(d)}$ with respect to the $W_1^{(d)}$-seminorm. Next, given $1\leq q<\infty$, we will write $B_q^{(d)}$ for the $W_q^{(d)}$-seminorm of the trigonometric polynomials $\mathcal{T}^{(d)}$; the elements of $B_q^{(d)}$ will be called \textit{$q$-Besicovitch sequences}. In other words, if $\alpha=(\alpha_\textbf{n})_{\textbf{n}\in\mathbb{N}_0^d}\in B_q^{(d)}$, then for every $\epsilon>0$ there exists a trigonometric polynomial $P_\epsilon$ such that
$$\left(\limsup_{\textbf{n}\to\infty}\frac{1}{|\textbf{n}|}\sum_{\textbf{k}=0}^{\textbf{n}-1}|\alpha_{\textbf{k}}-P_\epsilon(\textbf{k})|^q\right)^{\frac{1}{q}}<\epsilon.$$
Note that $\alpha\in B_q^{(d)}$ implies that $\alpha\in B_1^{(d)}$ as well. Elements of $B_\infty^{(d)}:=B_1^{(d)}\cap W_\infty$ are \textit{bounded Besicovitch sequences}. It is known that $B_1^{(d)}\cap W_\infty^{(d)}=B_r^{(d)}\cap W_\infty^{(d)}$ for every $1\leq r<\infty$ (see \cite[Theorem 3.1]{jo}), so the choice of $r=1$ does not make any difference.

A sequence $\alpha=(\alpha_{\textbf{n}})_{\textbf{n}\in\mathbb{N}_0^d}\in W_1^{(d)}$ is called a \textit{Hartman sequence} if
$$\lim_{\textbf{n}\to\infty}\frac{1}{|\textbf{n}|}\sum_{\textbf{k}=0}^{\textbf{n}-1}\alpha_{\textbf{k}}\underline{\lambda}^{\textbf{k}}\text{ exists for every }\underline{\lambda}\in\mathbb{T}^d,$$ and let $H^{(d)}\subset W_1^{(d)}$ denote the space of all Hartman sequences. It is known that $H^{(d)}$ is closed in $W_1^{(d)}$ and that $B_q^{(d)}\subset H^{(d)}\cap W_q^{(d)}$ for every $1\leq q\leq\infty$. We will only be concerned with the single parameter case $d=1$.

Numerous examples of \textit{bounded} Besicovitch sequences were given in \cite[Section 3]{bl}. In \cite{lot} it was shown that, if $(X,\mu)$ is a probability space, $\phi:X\to X$ is an ergodic $\mu$-preserving transformation, $1<q<\infty$, and $f\in L_q(X,\mu)$, then $(f(\phi^k(x)))_{k=0}^{\infty}\in W_q^{(1)}\cap H^{(1)}$ for $\mu$-a.e. $x\in X$ by the Birkhoff Individual Ergodic Theorem and the Wiener-Wintner Ergodic Theorem. Furthermore, if $\phi$ is not weakly mixing and $f$ is not constant $\mu$-a.e., then this sequence need not be a Besicovitch sequence. In \cite{comli2} the authors also proved that some sequences induced by admissible processes are bounded Besicovitch.

For ease of reference, we will present a method of constructing $q$-Besicovitch sequences in the single parameter case. The method is very similar to that of \cite[Theorem 3.19(2)]{bl}, with the main difference being that the sequences obtained are not necessarily bounded, and we limit ourselves to the particular case of an irrational rotation on $\mathbb{T}$. However, this important case should also help motivate the sequence definition of trigonometric polynomials.

\begin{ex}\label{e21}
Let $m$ denote the Lebesgue measure on $\mathbb{T}$. Fix $1<q<\infty$ and $f\in L_q(\mathbb{T},m)$. Let $\mu\in\mathbb{T}$ not be a root of unity (i.e. $\mu^n\neq1$ for every $n\in\mathbb{Z}\setminus\{0\}$), and define $\phi_{\mu}:\mathbb{T}\to\mathbb{T}$ by $\phi_{\mu}(\lambda)=\mu\lambda$. Then $\phi_{\mu}$ is an ergodic $m$-preserving transformation, so that there exists $F\subseteq\mathbb{T}$ with $m(F)=1$ such that $(f(\phi_{\mu}^{k}(\lambda)))_{k=0}^{\infty}\in W_q^{(1)}\cap H^{(1)}$ for every $\lambda\in F$ by \cite{lot} as stated above. We will show that $(f(\phi_{\mu}^{k}(\lambda)))_{k=0}^{\infty}\in B_q^{(1)}$ for $m$-a.e. $\lambda\in\mathbb{T}$.

Fix $N\in\mathbb{N}$. Since polynomials are dense in $L_q(\mathbb{T},m)$, there exists $g_N\in L_\infty(\mathbb{T},m)$ such that $\|f-g_N\|_q<\frac{1}{N}$ and $g_N(\lambda)=\sum_{j=0}^{\ell}\alpha_{j}\lambda^j$ for some $\alpha_0,...,\alpha_{\ell}\in\mathbb{C}$ and every $\lambda\in\mathbb{T}$. Since $\phi_{\mu}$ is ergodic, by applying Birkhoff's individual ergodic theorem to $|f-g_N|^q$, there exists a set $E_N\subseteq\mathbb{T}$ with $m(E_N)=1$ such that for every $\lambda\in E_N$,
$$\frac{1}{N^q}>\|f-g_N\|_q^q=\int_{\mathbb{T}}|f(\lambda)-g_N(\lambda)|^qdm(\lambda)=\lim_{n\to\infty}\frac{1}{n}\sum_{k=0}^{n-1}|f(\phi_{\mu}^k(\lambda))-g_N(\phi_{\mu}^k(\lambda))|^q.$$
Expanding the expression for $g_N(\phi_{\mu}^k(\lambda))$ shows that
$$g_N(\phi^k(\lambda))=\sum_{j=0}^{\ell}\alpha_{j}(\phi_{\mu}^k(\lambda))^j=\sum_{j=0}^{\ell}\alpha_j\lambda^j(\mu^j)^k.$$ Defining $P_{N,\lambda}:\mathbb{Z}\to\mathbb{C}$ by $P_{N,\lambda}(k)=g_N(\phi_{\mu}^k(\lambda))$, we find that $(P_{N,\lambda}(k))_{k=0}^{\infty}$ is a trigonometric polynomial satisfying
$$\frac{1}{N}>\left(\lim_{n\to\infty}\frac{1}{n}\sum_{k=0}^{n-1}|f(\phi_{\mu}^k(\lambda))-P_{N,\lambda}(k)|^q\right)^{1/q}.$$
If $E=F\cap\left(\bigcap_{N=1}^{\infty}E_N\right)$, then $m(E)=1$ and the above arguments then imply that if $\lambda\in E$, then $(P_{N,\lambda}(k))_{k=0}^{\infty}\to(f(\phi_\mu^k(\lambda)))_{k=0}^{\infty}$ as $N\to\infty$ with respect to $\|\cdot\|_{W_q^{(1)}}$. Therefore, if $1<q<\infty$, then $(f(\phi_{\mu}^{k}(\lambda)))_{k=0}^{\infty}\in B_q^{(1)}$ for every $f\in L_q(\mathbb{T},m)$ and for $m$-a.e. $\lambda\in\mathbb{T}$.
\end{ex}

\subsection{Noncommutative Ergodic Theorems}
For $1\leq p\leq\infty$, $\textbf{T}=(T_1,...,T_d)\in DS^+(\mathcal{M},\tau)^d$, $\textbf{n}=(n_1,...,n_d)\in\mathbb{N}^d$, $x\in L_p(\mathcal{M},\tau)$, and $\alpha=(\alpha_\textbf{k})_{\textbf{k}\in\mathbb{N}_0^d}\in W_1^{(d)}$, we will write 
$$M_\textbf{n}^\alpha(\textbf{T})(x):=\frac{1}{|\textbf{n}|}\sum_{\textbf{k}=0}^{\textbf{n}-1}\alpha_{\textbf{k}}\textbf{T}^{\textbf{k}}(x)=\frac{1}{n_1\cdots n_d}\sum_{k_1=0}^{n_1-1}\cdots\sum_{k_d=0}^{n_d-1}\alpha_{k_1,...,k_d}T_1^{k_1}\cdots T_d^{k_d}(x)$$ for the $\textbf{n}$-th $\alpha$-weighted multiparameter average of $x$ with respect to $T_1,...,T_d$. We will write $M_\textbf{n}(\textbf{T})$ to denote the corresponding usual averages with $\alpha_\textbf{n}=1$ for every $\textbf{n}\in\mathbb{N}_0^d$.

The notion of pointwise convergence doesn't necessarily make sense as is in the noncommutative setting. However, after a minor modification a noncommutative analogue of it does exist. As such, we will replace a.e. convergence with the following definitions inspired by Egorov's Theorem.

Let $(x_\textbf{n})_{\textbf{n}\in\mathbb{N}_0^d}\subseteq L_0(\mathcal{M},\tau)$ and $x\in L_0(\mathcal{M},\tau)$. Say that $x_\textbf{n}\to x$ bilaterally almost uniformly (b.a.u.) as $\textbf{n}\to\infty$ if, for every $\epsilon>0$, there exists $e\in\mathcal{P}(\mathcal{M})$ such that
$$\tau(e^\perp)\leq\epsilon \, \text{ and } \, \lim_{\textbf{n}\to\infty}\|e(x_\textbf{n}-x)e\|_\infty=0.$$
We will say that $x_\textbf{n}\to x$ almost uniformly (a.u.) if the limit $\lim_{\textbf{n}\to\infty}\|e(x_\textbf{n}-x)e\|_\infty$ can be replaced by $\lim_{\textbf{n}\to\infty}\|(x_\textbf{n}-x)e\|_\infty$.

Let $(\mathcal{M},\tau)$ be a semifinite von Neumann algebra and $(X,\|\cdot\|)$ be a normed space. A family of additive maps $S_i:X\to L_0(\mathcal{M},\tau)$, $i\in I$ (where $I$ is an arbitrary index set), is \textit{bilaterally uniformly equicontinuous in measure (b.u.e.m.) at $0$ on $(X,\|\cdot\|)$} if, for every $\epsilon,\delta>0$, there exists $\gamma>0$ such that $\|x\|<\gamma$ implies the existence of $e\in\mathcal{P}(\mathcal{M})$ such that
$$\tau(e^\perp)\leq\epsilon \, \text{ and } \, \sup_{i\in I}\|eS_i(x)e\|_\infty\leq\delta.$$
Similarly, the sequence is \textit{uniformly equicontinuous in measure (u.e.m.) at $0$ on $(X,\|\cdot\|)$} if $\sup_{i\in I}\|eS_i(x)e\|_\infty$ can be replaced by $\sup_{i\in I}\|S_i(x)e\|_\infty$.

Given $1\leq p<\infty$ and an index set $I$ (which may be uncountable), a family $S=(S_i)_{i\in I}$ of maps from $L_p(\mathcal{M},\tau)$ to $L_0(\mathcal{M},\tau)$ is of weak type $(p,p)$ with constant $C>0$ if, for every $\lambda>0$ and $x\in L_p(\mathcal{M},\tau)$, there exists $e\in\mathcal{P}(\mathcal{M})$ such that
$$\tau(e^\perp)\leq\frac{C^p\|x\|_p^p}{\lambda^p}\text{ and }\sup_{i\in I}\|eS_i(x)e\|_\infty\leq\lambda.$$

Note that $S=(S_n)_{i\in I}$ being weak type $(p,p)$ implies that it is b.u.e.m. at $0$ on $(L_p,\|\cdot\|_p)$ (use $\lambda=\delta$ and $\gamma=\epsilon^{\frac{1}{p}}\delta/C$).

Yeadon's noncommutative maximal ergodic inequality \cite[Theorem 2.1]{ye} will play an important role in this paper. For convenience, we state it and other useful results that will be used below.
\begin{teo}(Cf. \cite[Theorem 2.1]{ye})
Let $(\mathcal{M},\tau)$ be a semifinite von Neumann algebra and $T\in DS^+(\mathcal{M},\tau)$. Then the averages $(M_n(T))_{n=1}^{\infty}$ are of weak type $(1,1)$ with constant $C>0$. In particular, for every $x\in L_1^+$ and $\lambda>0$ there exists $e\in\mathcal{P}(\mathcal{M})$ such that
$$\tau(e^\perp)\leq\frac{\|x\|_1}{\lambda} \ \text{ and } \ \sup_{n}\|eM_n(T)(x)e\|_\infty\leq\lambda.$$
\end{teo}



\begin{pro}(Cf. \cite[Proposition 3.1]{cl2})\label{p22}
Let $(\mathcal{M},\tau)$ be a semifinite von Neumann algebra and $(X,\|\cdot\|)$ be a Banach space. Let $A_n:X\to L_0(\mathcal{M},\tau)$ be a family of additive maps. If $(A_n)_{n=1}^{\infty}$ is b.u.e.m. (u.e.m.) at zero on $(X,\|\cdot\|)$, then the set
$$\{x\in X:(A_n(x))_{n=1}^{\infty}\text{ converges b.a.u. (respectively, a.u.)}\}$$ is a closed subspace of $X$. 
\end{pro}

Finally, we will make use of the following reduction of order technique of Brunel. 
This will allow us to bound multiparameter averages by the averages of a single operator. 
Although originally proven in the commutative setting (see \cite{kr} for a proof), it was noted in \cite[Theorem 5.2]{cl2} that the proof extends to the noncommutative setting as well.

\begin{teo}\label{t22}
Let $(\mathcal{M},\tau)$ be a semifinite von Neumann algebra. Then for any $d\in\mathbb{N}$ with $d>1$ there exists a constant $\chi_d>0$ and a family $\{a_{\textbf{n}}>0:\textbf{n}\in\mathbb{N}_0^d\}$ with $\sum_{\mathbf{n}}a_{\mathbf{n}}=1$ such that, for any family $T_1,...,T_d:L_1(\mathcal{M},\tau)\to L_1(\mathcal{M},\tau)$ of commuting positive contractions, the operator $S=\sum_{(n_1,...,n_d)\in\mathbb{N}_0^d}a_{n_1,...,n_d}T_1^{n_1}\cdots T_d^{n_d}$ satisfies
$$\frac{1}{n^d}\sum_{k_1=0}^{n-1}\cdots\sum_{k_d=0}^{n-1}T_1^{k_1}\cdots T_d^{k_d}(x)\leq\frac{\chi_d}{n_d}\sum_{j=0}^{n_d-1}S^j(x)$$ for every $n\in\mathbb{N}$ and $x\in L_1(\mathcal{M},\tau)^+$, where $n_d\in\mathbb{N}$ depends on $n$ and $d$.
\end{teo}

\section{Convergence of Weighted Averages}\label{s3}

\subsection{Maximal Ergodic Inequalities} 

In this section, we will prove that the weighted averages of $T\in DS^+(\mathcal{M},\tau)$ are of weak type $(p,p)$ for weights in $W_q^{(1)}$, where $1<p,q<\infty$ and $\frac{1}{p}+\frac{1}{q}=1$. We will also prove a multiparameter version of this result for commuting operators $T_1,...,T_d\in DS^+(\mathcal{M},\tau)$ so long as the indices remain in a sector of $\mathbb{N}_0^d$. 



\begin{lm}\label{l31}
Let $(\mathcal{M},\tau)$ be a semifinite von Neumann algebra, $S_k:\mathcal{M}\to\mathcal{M}$ be positive contractions for every $k\in\mathbb{N}_0$, and $x\in \mathcal{M}^+$ be fixed. Let $1<p,q<\infty$ be such that $\frac{1}{p}+\frac{1}{q}=1$. If $\alpha=(\alpha_k)_{k=0}^{\infty}\subset[0,\infty)$, then for every $n\in\mathbb{N}$ we have 
$$
\frac{1}{n}\sum_{k=0}^{n-1}\alpha_kS_k(x)
\leq
\left(\frac{1}{n}\sum_{k=0}^{n-1}\alpha_k^q\right)^{\frac{1}{q}}\left(\frac{1}{n}\sum_{k=0}^{n-1}S_k(x^p)\right)^{\frac{1}{p}}.
$$
\end{lm}
\begin{proof}
First we will assume that $1<p\leq2$. Observe that, if $\beta_k,\gamma_k\geq0$ for every $k=0,...,n-1$, then Lemma \ref{l21} implies that
$$
\frac{1}{n}\sum_{k=0}^{n-1}\beta_k\gamma_kS_k(x)
\leq
\left(\frac{1}{n}\sum_{k=0}^{n-1}\beta_k^q\right)^{\frac{1}{q}}\left(\frac{1}{n}\sum_{k=0}^{n-1}\gamma_k^pS_k(x)^p\right)^{\frac{1}{p}}.$$ Since $0\leq t\mapsto t^p$ is operator convex for $1<p\leq2$, $x\in\mathcal{M}^+$, and $S_k$ is a positive contraction of $\mathcal{M}$, it follows that $S_k(x)^p\leq S_k(x^p)$ for every $k=0,...,n-1$. Since $\mathcal{M}^+$ is a positive cone, and since $0\leq t\mapsto t^{\frac{1}{p}}$ is operator monotone, it follows that
$$
\left(\frac{1}{n}\sum_{k=0}^{n-1}\gamma_k^pS_k(x)^p\right)^{\frac{1}{p}}
\leq
\left(\frac{1}{n}\sum_{k=0}^{n-1}\gamma_k^pS_k(x^p)\right)^{\frac{1}{p}}.$$
Therefore
$$
\frac{1}{n}\sum_{k=0}^{n-1}\beta_k\gamma_kS_k(x)
\leq
\left(\frac{1}{n}\sum_{k=0}^{n-1}\beta_k^q\right)^{\frac{1}{q}}\left(\frac{1}{n}\sum_{k=0}^{n-1}\gamma_k^pS_k(x^p)\right)^{\frac{1}{p}}.
$$
Since $\alpha_k\geq0$ for each $k$, one may use $\beta_k=\alpha_k$ and $\gamma_k=1$ to find that
$$\frac{1}{n}\sum_{k=0}^{n-1}\alpha_kS_k(x)\leq\left(\frac{1}{n}\sum_{k=0}^{n-1}\alpha_k^q\right)^{\frac{1}{q}}\left(\frac{1}{n}\sum_{k=0}^{n-1}S_k(x^p)\right)^{\frac{1}{p}}.$$

Now assume there is $m\in\mathbb{N}$ such that that the claim holds for every element of $(1,2^m]$, and let $p\in(2^m,2^{m+1}]$. If $p_1=\frac{p}{2}$, $p_2=2$, $q_1=\frac{p}{p-2}$, and $q_2=2$, then $p=p_1p_2$ and $\frac{1}{p_j}+\frac{1}{q_j}=1$ for $j=1,2$ (also note that $\frac{p}{2}\in(1,2^m]$). Letting $q:=((q_1)^{-1}+(q_2p_1)^{-1})^{-1}$, then $1<q<\infty$ and $\frac{1}{p}+\frac{1}{q}=1$.

The expression for $q$ may be rewritten as $\frac{q}{q_1}+\frac{q}{q_2p_1}=1$. Using $\beta_k=\alpha_k^{\frac{q}{q_1}}$ and $\gamma_k=\alpha_k^{\frac{q}{q_2p_1}}$, by the induction hypothesis we find that
\begin{align*}
\frac{1}{n}\sum_{k=0}^{n-1}\alpha_kS_k(x)
=&\frac{1}{n}\sum_{k=0}^{n-1}\alpha_k^{\frac{q}{q_1}}\alpha_k^{\frac{q}{q_2p_1}}S_k(x) \\
\leq&\left(\frac{1}{n}\sum_{k=0}^{n-1}\left(\alpha_k^{\frac{q}{q_1}}\right)^{q_1}\right)^{\frac{1}{q_1}}\left(\frac{1}{n}\sum_{k=0}^{n-1}\left(\alpha_k^{\frac{q}{q_2p_1}}\right)^{p_1}S_k(x^{p_1})\right)^{\frac{1}{p_1}} \\
=&\left(\frac{1}{n}\sum_{k=0}^{n-1}\alpha_k^q\right)^{\frac{1}{q_1}}\left(\frac{1}{n}\sum_{k=0}^{n-1}\alpha_k^{\frac{q}{q_2}}S_k(x^{p_1})\right)^{\frac{1}{p_1}}.
\end{align*}

Applying the argument for the base case again to the second term of this product using $p_2=q_2=2$, we find that
$$\frac{1}{n}\sum_{k=0}^{n-1}\alpha_k^{\frac{q}{q_2}}S_k(x^{p_1})
\leq
\left(\frac{1}{n}\sum_{k=0}^{n-1}\left(\alpha_k^{\frac{q}{q_2}}\right)^{q_2}\right)^{\frac{1}{q_2}}\left(\frac{1}{n}\sum_{k=0}^{n-1}S_k(x^{p_1p_2})\right)^{\frac{1}{p_2}}.$$

Since the map $0\leq t\mapsto t^{\frac{1}{p_1}}$ is operator monotone, with the above inequality we find that
\begin{align*}
\frac{1}{n}\sum_{k=0}^{n-1}\alpha_kS_k(x)
\leq&\left(\frac{1}{n}\sum_{k=0}^{n-1}\alpha_k^q\right)^{\frac{1}{q_1}}\left(\frac{1}{n}\sum_{k=0}^{n-1}\alpha_k^{\frac{q}{q_2}}S_k(x^{p_1})\right)^{\frac{1}{p_1}} \\
\leq&\left(\frac{1}{n}\sum_{k=0}^{n-1}\alpha_k^q\right)^{\frac{1}{q_1}}\left(\frac{1}{n}\sum_{k=0}^{n-1}\alpha_k^q\right)^{\frac{1}{q_2p_1}}\left(\frac{1}{n}\sum_{k=0}^{n-1}S_k(x^{p_1p_1})\right)^{\frac{1}{p_1p_2}} \\
=&\left(\frac{1}{n}\sum_{k=0}^{n-1}\alpha_k^q\right)^{\frac{1}{q}}\left(\frac{1}{n}\sum_{k=0}^{n-1}S_k(x^p)\right)^{\frac{1}{p}}.
\end{align*}
The claim then follows for every $1<p<\infty$ by induction on $m\in\mathbb{N}_0$.
\end{proof}

\begin{rem}
Although this lemma is similar to Lemma \ref{l21} above in essence, it is not a direct consequence of it. If $1<p\leq2$, it is indeed an immediate consequence of Lemma \ref{l21} since $T^k(x)^p\leq T^k(x^p)$ due to the operator convexity of the map $0\leq t\mapsto t^p$. On the other hand, as mentioned in Section \ref{s2}, the map $0\leq t\mapsto t^p$ is no longer operator convex when $2<p<\infty$, meaning we can no longer guarantee that $S^k(x)^p\leq S^k(x^p)$ for $x\in\mathcal{M}^+$. Due to this, one needs to utilize different techniques to extend it into this range. The author would like to thank Dr. L\'{e}onard Cadilhac for his suggestion that enabled us to extend the result to hold for every $1<p<\infty$ instead of for only $1<p\leq2$ (and only with special cases for $2<p<\infty$). 
\end{rem}

\begin{rem}
We note that this result holds for multiparameter weighted averages as well by enumerating $([0,n_1]\times\cdots\times[0,n_d])\cap\mathbb{N}_0$ if $\textbf{n}=(n_1,...,n_d)\in\mathbb{N}_0^d$ and choosing the $S_k$'s as the product of $T_i$'s with powers according to this enumeration 
For example, if $T_1,T_2\in DS^+(\mathcal{M},\tau)$ and $\textbf{n}=(2,3)$, then one could use $n=6$, $S_0=Id_{\mathcal{M}}$, $S_1=T_1$, $S_2=T_2$, $S_3=T_1T_2$, $S_4=T_2^2$, and $S_5=T_1T_2^2$.
\end{rem}

We now use this to prove a multiparameter weak type $(p,p)$ maximal ergodic inequality for commuting operators with weights in $W_q^{(d)}$.

\begin{teo}\label{t31}
Let $(\mathcal{M},\tau)$ be a semifinite von Neumann algebra, $d\in\mathbb{N}$, $\textbf{T}=(T_1,...,T_d)\in DS^+(\mathcal{M},\tau)^d$ consist of commuting operators, and $1<p,q<\infty$ be such that $\frac{1}{p}+\frac{1}{q}=1$. 
Then, for every $C\geq1$, the family $\left(|\alpha|_{W_q}^{-1}M_\textbf{n}^\alpha(\textbf{T})\right)_{(\textbf{n},\alpha)\in \textbf{N}_C^{(d)}\times W_q^{(d)}}$ of multiparameter $W_q^{(d)}$-weighted averages of $(T_1,...,T_d)$ are of weak type $(p,p)$ with constant $4^{2+\frac{1}{p}}(C^d\chi_d)^{\frac{1}{p}}$, where $\chi_d$ is the constant from Theorem \ref{t22} when $d\geq2$, or $\chi_{d}=\chi_{1}=1$ when $d=1$. 
\end{teo}
\begin{proof}

Since $L_1\cap\mathcal{M}$ is dense in the measure topology of $L_0(\mathcal{M},\tau)$ and in the norm topology of $L_p(\mathcal{M},\tau)$, without loss of generality we may assume that $x\in L_1\cap\mathcal{M}$. Since each $x\in L_1\cap\mathcal{M}$ can be written as the sum of four elements of $L_1\cap\mathcal{M}^+$ whose $L_p$-norm does not exceed that of $x$, we may first prove the claim for $x\in L_1\cap\mathcal{M}^+$.

Assume $\lambda>0$, $x\in L_1\cap\mathcal{M}^+$, $C\geq1$, and $\alpha=(\alpha_\textbf{k})_{\textbf{k}\in\mathbb{N}_0^d}\in (W_q^{(d)})^+$. Then, since $x\in\mathcal{M}^+$ and since each $T_i\in DS^+(\mathcal{M},\tau)$, it follows by Lemma \ref{l31} that, for every $\textbf{n}\in \textbf{N}_C^{(d)}$,
$$\frac{1}{|\textbf{n}|}\sum_{\textbf{k}=0}^{\textbf{n}-1}\alpha_{\textbf{k}}\textbf{T}^\textbf{k}(x)
\leq\left(\frac{1}{|\textbf{n}|}\sum_{\textbf{k}=0}^{\textbf{n}-1}\alpha_\textbf{k}^q\right)^{\frac{1}{q}}\left(\frac{1}{|\textbf{n}|}\sum_{k=0}^{n-1}\textbf{T}^\textbf{k}(x^p)\right)^{\frac{1}{p}}
\leq|\alpha|_{W_{q,C}^{(d)}}(M_\textbf{n}(\textbf{T})(x^p))^{\frac{1}{p}}.$$

If $d=1$, then without loss of generality we may use $C=1$ and $\chi_{1}=1$, so that $(C^d\chi_{1})^{\frac{1}{p}}=1$ in the constant for the maximal inequality. Also, we will set $S=T_1$, noting that $S\in DS^+(\mathcal{M},\tau)$. Finally, use $(m_{\textbf{n}})_d=\textbf{n}$ for each $\textbf{n}\in\mathbb{N}$.

If $d\geq2$, let $S\in DS^+(\mathcal{M},\tau)$ and $\chi_{d}>0$ be the operator and constant from Theorem \ref{t22} corresponding to $(T_1,...,T_d)$. If $\textbf{n}=(n_1,...,n_d)\in\textbf{N}_C^{(d)}$ and $m_{\textbf{n}}=\max\{n_1,...,n_d\}$, then we find that $m_\textbf{n}\leq Cn_j$ for each $j=1,...,d$, which implies $\frac{1}{n_j}\leq\frac{C}{m_{\textbf{n}}}$ for every such $j$. Using this, and extending the upper index on each sum to $m_{\textbf{n}}-1$ (which satisfies the operator inequality since $T_1^{k_1}\cdots T_d^{k_d}(x^p)\geq0$ for every $k_1,...,k_d\in\mathbb{N}_0$), by Theorem \ref{t22} we find that
\begin{align*}
\frac{1}{|\textbf{n}|}\sum_{\textbf{k}=0}^{\textbf{n}-1}\textbf{T}^{\textbf{k}}(x^p)
=&\frac{1}{n_1\cdots n_d}\sum_{k_1=0}^{n_1-1}\cdots\sum_{k_d=0}^{n_d-1}T_1^{k_1}\cdots T_d^{k_d}(x^p) \\
\leq&\frac{C^d}{m_{\textbf{n}}^d}\sum_{k_1=0}^{m_{\textbf{n}}-1}\cdots\sum_{k_d=0}^{m_{\textbf{n}}-1}T_1^{k_1}\cdots T_d^{k_d}(x^p) \\
\leq&\frac{C^d\chi_d}{(m_\textbf{n})_d}\sum_{j=0}^{(m_{\textbf{n}})_d-1}S^j(x^p)
=C^d\chi_d M_{(m_\textbf{n})_d}(S)(x^p).
\end{align*}

In either case, since $x\in L_1\cap\mathcal{M}^+$ implies  $x^p\in L_1(\mathcal{M},\tau)^+$, by Yeadon's weak type $(1,1)$ maximal ergodic inequality applied to $(M_n(S))_{n=1}^{\infty}$ there exists $e\in\mathcal{P}(\mathcal{M})$ satisfying both
$$\tau(e^\perp)\leq\frac{16^pC^d\chi_d\|x^p\|_1}{\lambda^p}=\frac{16^pC^d\chi_d\|x\|_p^p}{\lambda^p} \ \text{ and } \ \sup_{n\in\mathbb{N}}\|eM_n(S)(x^p)e\|_\infty\leq\frac{\lambda^p}{16^pC^d\chi_d}.$$
Fixing $n$, the fact that $M_n(S)(x^p)\in\mathcal{M}^+$ implies (from the norm inequality) that
$$eM_n(S)(x^p)e\leq\frac{\lambda^p}{16^p C^d\chi_d}e,$$ and the operator monotonicity of $0\leq t\mapsto t^{\frac{1}{p}}$ implies that
$$(eM_n(S)(x^p)e)^{\frac{1}{p}}\leq\frac{\lambda}{16(C^d\chi_d)^{\frac{1}{p}}}e.$$ Reapplying norms and taking the supremum over $\mathbb{N}$, it follows that
$$\sup_{n\in\mathbb{N}}\|(eM_n(S)(x^p)e)^{\frac{1}{p}}\|_\infty\leq\frac{\lambda}{16(C^d\chi_d)^{\frac{1}{p}}}.$$ Using Equation (3.18) of \cite{mei}, in the main inequality it follows that
$$e(M_n(S)(x^p))^{\frac{1}{p}}e\leq(eM_n(S)(x^p)e)^{\frac{1}{p}}.$$ Therefore, we find that
\begin{align*}
\left\|e\left(\frac{1}{|\textbf{n}|}\sum_{\textbf{k}=0}^{\textbf{n}-1}\alpha_{\textbf{k}}\textbf{T}^{\textbf{k}}(x)\right)e\right\|_\infty
\leq&
|\alpha|_{W_{q,C}^{(d)}}\|e(M_\textbf{n}(\textbf{T})(x^p))^{\frac{1}{p}}e\|_\infty \\
\leq&
(C^d\chi_{d})^{\frac{1}{p}}|\alpha|_{W_{q,C}^{(d)}}\left\|e(M_{(m_{\textbf{n}})_d}(S)(x^p)^{\frac{1}{p}})e\right\|_\infty \\
\leq&
(C^d\chi_{d})^{\frac{1}{p}}|\alpha|_{W_{q,C}^{(d)}}\left\|(eM_{(m_{\textbf{n}})_d}(S)(x^p)e)^{\frac{1}{p}}\right\|_\infty \\
\leq&
(C^d\chi_{d})^{\frac{1}{p}}|\alpha|_{W_{q,C}^{(d)}}\frac{\lambda}{16(C^d\chi_d)^{\frac{1}{p}}}
=|\alpha|_{W_{q,C}^{(d)}}\frac{\lambda}{16}.
\end{align*}
Dividing both sides by $|\alpha|_{C,W_q^{(d)}}$ (remembering that $\frac{0}{0}:=0$ when $\alpha=\textbf{0}$),
we obtain
$$\left\|e\left(\frac{M_\textbf{n}^\alpha(\textbf{T})}{|\alpha|_{W_{q,C}^{(d)}}}\right)(x)e\right\|_\infty\leq\frac{\lambda}{16}.$$
Finally, since $\alpha\in (W_q^{(d)})^+$ and $\textbf{n}\in \textbf{N}_C^{(d)}$ were arbitrary, it follows that
$$\sup_{(\textbf{n},\alpha)\in \textbf{N}_C^{(d)}\times (W_q^{(d)})^+}\left\|e\left(\frac{M_\textbf{n}^\alpha(\textbf{T})}{|\alpha|_{W_{q,C}^{(d)}}}\right)(x)e\right\|_\infty\leq\frac{\lambda}{16}.$$

Now, assume $\alpha\in W_q^{(d)}$, and let $\alpha_0,...,\alpha_3\in (W_q^{(d)})^+\cup\{\textbf{0}\}$ be sequences such that $\alpha=\sum_{j=0}^{3}i^j\alpha_j$ and $|\alpha_j|_{W_{q,C}^{(d)}}\leq|\alpha|_{W_{q,C}^{(d)}}$ for each $j=0,...,3$. By the triangle inequality, one finds that
\begin{align*}
\left\|e\left(\frac{M_\textbf{n}^\alpha(\textbf{T})}{|\alpha|_{W_{q,C}^{(d)}}}\right)(x)e\right\|_\infty
\leq&
\sum_{j=0}^{3}\frac{1}{|\alpha|_{W_{q,C}^{(d)}}}\left\|eM_n^{\alpha_j}(T)(x)e\right\|_\infty \\
\leq&
\sum_{j=0}^{3}\frac{|\alpha_j|_{W_{q,C}^{(d)}}}{|\alpha|_{W_{q,C}^{(d)}}}\frac{\lambda}{16}\leq\frac{\lambda}{4}.
\end{align*}

Assume now that $x\in L_1\cap\mathcal{M}$, and write $x=\sum_{j=0}^{3}i^jx_j$, where $x_j\in L_1\cap\mathcal{M}^+$ and $\|x_j\|_p\leq\|x\|_p$ for each $j=0,...,3$. Then there exists $e_0,...,e_3\in\mathcal{P}(\mathcal{M})$ such that
$$\tau(e_j^\perp)\leq\frac{16^p C^d\chi_d\|x_j\|_p^p}{\lambda^p} \ \text{ and } \ \sup_{(\textbf{n},\alpha)\in\textbf{N}_C^{(d)}\times W_q^{(d)}}\left\|e_j\left(\frac{M_\textbf{n}^\alpha(\textbf{T})}{|\alpha|_{W_{q,C}^{(d)}}}(x_j)\right)e_j\right\|_\infty\leq\frac{\lambda}{4}.$$
Let $e=\bigwedge_{j=0}^{3}e_j$. Then $$\tau(e^\perp)
\leq\sum_{j=0}^{3}\dfrac{16^p C^d\chi_d\|x_j\|_p^p}{\lambda^p}
\leq4\cdot\frac{16^pC^d\chi_d\|x\|_p^p}{\lambda^p}
=\dfrac{(4^{2+\frac{1}{p}}(C^d\chi_d)^{\frac{1}{p}})^{p} \|x\|_p^p}{\lambda^p}$$ and
\begin{align*}
\sup_{(\textbf{n},\alpha)\in\textbf{N}_C^{(d)}\times W_q^{(d)}}\left\|e\left(\frac{M_\textbf{n}^\alpha(\textbf{T})}{|\alpha|_{W_{q,C}^{(d)}}}(x)\right)e\right\|_\infty
\leq&
\sum_{j=0}^{3}\sup_{(\textbf{n},\alpha)\in\textbf{N}_C^{(d)}\times W_{q}^{(d)}}\left\|e_j\left(\frac{M_\textbf{n}^\alpha(\textbf{T})}{|\alpha|_{W_{q,C}^{(d)}}}(x_j)\right)e_j\right\|_\infty \\
\leq&\sum_{j=0}^{3}\frac{\lambda}{4}=\lambda.
\end{align*}
Since $\lambda>0$ and $x\in L_1\cap\mathcal{M}$ were arbitrary, the result follows.
\end{proof}

As far as the author is aware, this result is new in the single parameter setting. That is why we state it below separately for ease of reference and notation.

\begin{cor}
Let $(\mathcal{M},\tau)$ be a semifinite von Neumann algebra, $T\in DS^+(\mathcal{M},\tau)$, $1<p,q<\infty$ satisfy $\frac{1}{p}+\frac{1}{q}=1$, and $\alpha\in W_q$. Then the weighted averages $\Big(\frac{1}{|\alpha|_{W_{q,1}^{(1)}}}M_n^\alpha(T)\Big)_{(n,\alpha)\in\mathbb{N}\times W_q^{(1)}}$ are of weak type $(p,p)$ with constant $4^{2+\frac{1}{p}}$. In particular, the sequence $(M_n^\alpha(T))_{n=1}^{\infty}$ is of weak type $(p,p)$ for every $\alpha\in W_q$.
\end{cor}

A version of this statement that is suitable to almost uniform convergence requires some additional restrictions. Namely, $\frac{1}{p}+\frac{1}{q}=1$ must be replaced by $\frac{2}{p}+\frac{1}{q}=1$, and the process uses one sequence at a time. This is due to the fact that it follows the argument in \cite[Proposition 4.1]{cl1}, and the use of Kadison's inequality and the b.u.e.m. result for $x^2\in L_{p/2}$ (when $x\in L_p^+\cap\mathcal{M}$). We obtain the result for weights in a $W_q^{(1)}$-space that is smaller then the one previously obtained for b.a.u. convergence results.

For our applications, considering only a single sequence at a time for the a.u. convergence results won't actually change any conclusions we obtain. For similar reasons, we will restrict ourselves to a single parameter in this extension. Due to all of this, we do not pursue any stronger (or multiparameter) versions of the a.u. extensions at this time.

\begin{pro}\label{p32}
Let $(\mathcal{M},\tau)$ be a semifinite von Neumann algebra, $T\in DS^+(\mathcal{M},\tau)$, $2<p<\infty$ and $1<q<\infty$ satisfying $\frac{2}{p}+\frac{1}{q}=1$, and $\alpha\in W_q^{(d)}$. 
Then the weighted averages $(M_n^\alpha(T))_{n=1}^{\infty}$ are u.e.m. at zero on $(L_p,\|\cdot\|_p)$.
\end{pro}
\begin{proof}
As in Theorem \ref{t31}, assume without loss of generality that $\alpha_n\geq0$ for every $n\geq0$, so that each $M_n^\alpha(T)$ is a positive map.

Assume $\epsilon,\delta>0$. By Theorem \ref{t31}, the averages $(M_n^\alpha(T))_{n=1}^{\infty}$ are b.u.e.m. at zero on $(L_{p/2},\|\cdot\|_{p/2})$; let $\gamma>0$ to be the value corresponding to $\epsilon$ and $\delta^2$ in that definition. Let $x\in L_p\cap\mathcal{M}^+$ be such that $\|x\|_p<\sqrt{\gamma}$. Then $\|x^2\|_{p/2}=\|x\|_p^2<\gamma$, so there exists $e\in\mathcal{P}(\mathcal{M})$ with
$$\tau(e^\perp)\leq\epsilon\text{ and }\sup_{n\in\mathbb{N}}\|eM_n^\alpha(T)(x^2)e\|_\infty\leq|\alpha|_{W_q}\delta^2.$$
Since $x\in\mathcal{M}^+$ and $M_n^\alpha(T)$ is positive map, Kadison's inequality implies that
$$M_n^\alpha(T)(x)^2\leq M_n^\alpha(T)(x^2), \, \text{ so that } \, eM_n^\alpha(T)(x)^2e\leq eM_n^\alpha(T)(x^2)e.$$
Hence
$$\|M_n^\alpha(T)(x)e\|_\infty^2=\|eM_n^\alpha(T)(x)^2e\|_\infty\leq\|eM_n^\alpha(T)(x^2)e\|_\infty\leq\delta^2,$$ so that
$\sup_{n}\|M_n^\alpha(T)(x)e\|_\infty\leq\delta.$ Since $x\in L_p\cap\mathcal{M}^+$ and $\epsilon,\delta>0$ were arbitrary, the conclusion follows by \cite[Theorem 3.2 and Lemma 4.1]{li1}.
\end{proof}

\subsection{$q$-Besicovitch sequences} 

We will now use Theorem \ref{t31} to prove that $q$-Besicovitch sequences can be used as weights for individual ergodic theorems regarding noncommutative $L_p$-spaces (with $\frac{1}{p}+\frac{1}{q}=1$). The process is standard, but we include it for the sake of completeness.

\begin{pro}\label{p33}
Let $(\mathcal{M},\tau)$ be a semifinite von Neumann algebra, $d\in\mathbb{N}$ (respectively, $d=1$), $\textbf{T}\in DS^+(\mathcal{M},\tau)^d$, and $x\in L_1(\mathcal{M},\tau)\cap\mathcal{M}$. Let $\mathcal{A}\subset W_1^{(d)}$, and let $\mathcal{C}\subset W_1^{(d)}$ denote the $W_1^{(d)}$-seminorm closure of $\mathcal{A}$. Then, if $(M_{\textbf{n}_k}^\alpha(\textbf{T})(x))_{k=0}^{\infty}$ converges b.a.u. (respectively, a.u.) for every $\alpha\in\mathcal{A}$ and every sequence $(\textbf{n}_k)_{k=0}^{\infty}$ in a sector of $\mathbb{N}^d$ with $\textbf{n}_k\to\infty$, then it does so for every $\alpha\in\mathcal{C}$.
\end{pro}
\begin{proof}
Assume $\alpha=(\alpha_\textbf{n})_{\textbf{n}\in\mathbb{N}_0^d}\in\mathcal{C}$ and $\epsilon>0$. Then there exists $\beta=(\beta_\textbf{n})_{\textbf{n}\in\mathbb{N}_0^d}\in\mathcal{A}$ such that
$$\limsup_{\textbf{n}\to\infty}\frac{1}{|\textbf{n}|}\sum_{\textbf{k}=0}^{\textbf{n}-1}|\alpha_\textbf{k}-\beta_\textbf{k}|=\|\alpha-\beta\|_{W_1^{(d)}}<\epsilon.$$
Let $N\in\mathbb{N}_0$ be such that this inequality holds for every $\textbf{n}=(n_1,...,n_d)$ with $\min\{n_1,...,n_d\}\geq N$. Assume $(\textbf{n}_k)_{k=0}^{\infty}\subset\textbf{N}_C^{(d)}$ satisfies $\textbf{n}_k\to\infty$ as $k\to\infty$ for some $C\geq1$. Then, since $\textbf{n}_k\to\infty$, there exists $N_\epsilon$ such that this inequality holds for every $k\geq N_\epsilon$, and so we also have that
$$\|M_{\textbf{n}_k}^\alpha(\textbf{T})(x)-M_{\textbf{n}_k}^\beta(\textbf{T})(x)\|_\infty
\leq\frac{1}{|\textbf{n}_k|}\sum_{\textbf{k}=0}^{\textbf{n}_k-1}|\alpha_\textbf{k}-\beta_\textbf{k}|\|T^\textbf{k}(x)\|_\infty
\leq\epsilon\|x\|_\infty.$$
Since $(M_{\textbf{n}_k}^\beta(\textbf{T})(x))_{k=0}^{\infty}$ converges b.a.u. (or a.u.) as $k\to\infty$ by assumption, and since $\epsilon>0$ was arbitrary, it follows by that $(M_{\textbf{n}_k}^\alpha(\textbf{T})(x))_{k=0}^{\infty}$ converges b.a.u. (respectively, a.u.) as well by \cite[Lemma 4.3]{cls}.
\end{proof}

Since $L_1\cap\mathcal{M}$ is dense in $L_p(\mathcal{M},\tau)$ for every $1\leq p<\infty$, we obtain the following by applying Propositions \ref{p22} and \ref{p33}.

\begin{cor}\label{c31}
Let $(\mathcal{M},\tau)$ be a semifinite von Neumann algebra, $d\in\mathbb{N}$ (respectively, $d=1$) $\textbf{T}=(T_1,...,T_d)\in DS^+(\mathcal{M},\tau)^d$, and $1<p,q<\infty$ (respectively, $2<p<\infty$).
Let $\mathcal{A}\subset W_q^{(d)}$, and let $\mathcal{C}\subset W_q^{(d)}$ be the $W_q^{(d)}$-seminorm closure of $\mathcal{A}$. If $\frac{1}{p}+\frac{1}{q}=1$ (respectively, $\frac{2}{p}+\frac{1}{q}=1$) and if $(M_{\textbf{n}_k}^\alpha(\textbf{T})(x))_{k=0}^{\infty}$ converges b.a.u. (respectively, a.u.) for every $\alpha\in\mathcal{A}$, $x\in L_p(\mathcal{M},\tau)$, and sequence $(\textbf{n}_k)_{k=0}^{\infty}$ in a sector of $\mathbb{N}_0^d$ with $\textbf{n}_k\to\infty$, then it does so for every $\alpha\in\mathcal{C}$ and $x\in L_p(\mathcal{M},\tau)$.
\end{cor}

It was shown in \cite{cls,cl1,cl2} that, given a semifinite von Neumann algebra $(M,\tau)$ with a separable predual, the averages $(M_n^\alpha(T)(x))_{n=1}^{\infty}$ converge b.a.u. for every $x\in L_p(\mathcal{M},\tau)$, $T\in DS^+(\mathcal{M},\tau)$, and bounded Besicovitch sequence $\alpha\in B_\infty^{(d)}$, where $1\leq p<\infty$, and that the convergence even occurs a.u. when $2\leq p<\infty$. For multiparameter averages, the same holds b.a.u. for every $1<p<\infty$ (see \cite{mmt}).


\begin{teo}\label{t32}
Let $(\mathcal{M},\tau)$ be a semifinite von Neumann algebra with a separable predual, $d\in\mathbb{N}$, $\textbf{T}=(T_1,...,T_d)\in DS^+(\mathcal{M},\tau)^d$ consist of commuting operators, $1<p,q<\infty$ satisfy $\frac{1}{p}+\frac{1}{q}=1$, and $\alpha\in B_q^{(d)}$ be a $q$-Besicovitch sequence. 
Then for every $x\in L_p(\mathcal{M},\tau)$, there exists $x_\alpha\in L_p(\mathcal{M},\tau)$ such that the weighted averages $(M_{\textbf{n}_k}^\alpha(\textbf{T})(x))_{k=0}^{\infty}$ converge b.a.u. to $x_\alpha$ along any sequence $(\textbf{n}_k)_{k=0}^{\infty}$ in a sector of $\mathbb{N}_0^d$ that tends to $\infty$ as $k\to\infty$. 
\end{teo}
\begin{proof}
For a fixed sequence $(\textbf{n}_k)_{k=0}^{\infty}$ tending to $\infty$ in a sector of $\mathbb{N}_0^d$, we know by \cite[Theorem 3.1]{mmt} that the averages $(M_{\textbf{n}_k}^\alpha(\textbf{T})(x))_{k=1}^{\infty}$ converge b.a.u. for every $x\in L_1\cap\mathcal{M}$ and every $\alpha\in \mathcal{T}^{(d)}$. 
Since, by definition, the closure of $\mathcal{T}^{(d)}$ with respect to the $W_q^{(d)}$-seminorm is equal to $B_q$, the b.a.u. convergence is obtained by Corollary \ref{c31}. The limit $x_\alpha$ is in $L_p(\mathcal{M},\tau)$ by \cite[Proposition 4.10]{ob2}.

With $x\in L_p$ fixed, let $(\textbf{m}_j)_{j=0}^{\infty}$ be another sequence that tends to $\infty$ in a sector of $\mathbb{N}_0^d$, and suppose that $M_{\textbf{m}_j}^{\alpha}(\textbf{T})(x)\to\widehat{x}$ b.a.u. as $j\to\infty$. If $C_1,C_2>0$ are such that $(\textbf{n}_k)_{k=0}^{\infty}\subset \textbf{N}_{C_1}^{(d)}$ and $(\textbf{m}_j)_{j=0}^{\infty}\subset \textbf{N}_{C_2}^{(d)}$, then by letting $C=\max\{C_1,C_2\}$ we see that the sequence $\textbf{w}=(\textbf{w}_m)_{m=0}^{\infty}=(\textbf{n}_0,\textbf{m}_0,\textbf{n}_1,\textbf{m}_1,...)$ is contained in $\textbf{N}_C^{(d)}$, so that it remains in a sector of $\mathbb{N}_0^d$ as well. Therefore, since $x_\alpha$ and $\widehat{x}$ are both subsequential limit points of $(M_{\textbf{w}_m}(\textbf{T})(x))_{m=0}^{\infty}$ with respect to b.a.u. convergence, and so with respect to convergence in measure, it must be that $\widehat{x}=x_\alpha$ since $L_0$ is Hausdorff. Hence, the limit $x_\alpha$ is independent of $(\textbf{n}_j)_{j=1}^{\infty}\subset \textbf{N}_C^{(d)}$.
\end{proof}

This result is new in the single parameter case as well, so we write it separately. 

\begin{cor}
Let $(\mathcal{M},\tau)$ be a semifinite von Neumann algebra with a separable predual, $T\in DS^+(\mathcal{M},\tau)$, $1<p,q<\infty$, and $\alpha\in B_q^{(1)}$. If $\frac{1}{p}+\frac{1}{q}=1$, then for every $x\in L_p(\mathcal{M},\tau)$ the weighted averages $(M_n^\alpha(T)(x))_{n=1}^{\infty}$ converge b.a.u. to some $x_\alpha\in L_p(\mathcal{M},\tau)$. If $2<p<\infty$ and $\frac{2}{p}+\frac{1}{q}=1$, then the convergence occurs a.u.
\end{cor}

\begin{rem}\label{r31}
It should be noted that the assumptions made above are not new in the context of multiparameter ergodic theorems. Namely, it is known that the operators $T_1,...,T_d$ do not need to commute for the a.e./(b.)a.u. convergence of the unweighted averages $(M_\textbf{n}(\textbf{T})(x))_{\textbf{n}\in\mathbb{N}_0^d}$ when $x\in L_p$ with $1<p<\infty$ in either the commutative or noncommutative setting (see \cite{jo} for the commutative setting and \cite{jx} for the noncommutative setting). However, in the commutative case, when $x\in L_1$ there are known counterexamples for this general setting, but the desired convergence does occur when the assumptions that we made throughout are satisfied; the proofs of the results end up also being similar as well.

We note that many improvements can be made to these results in the commutative setting. There, the convergence can be studied without involving sectors, and the $T_j$'s need not commute. Furthermore, each $T_j\in DS^+(\mathcal{M},\tau)$ can be replaced by $T_j:L_p(\mathcal{M},\tau)\to L_p(\mathcal{M},\tau)$ positive contractions only (without it necessarily being defined on some other $L_r$-spaces).

The proof of this version of the commutative result in \cite{jo} involves the dilation of operators. In the noncommutative setting, however, it was shown that this can't necessarily be done \cite{jlm}. It turns out that there exists a finite von Neumann algebra $(\mathcal{M},\tau)$ and positive contraction $T:L_p(\mathcal{M},\tau)\to L_p(\mathcal{M},\tau)$ that does not admit any dilation to an isometry on $L_p(\mathcal{N},\rho)$ for any semifinite von Neumann algebra $(\mathcal{N},\rho)$. Therefore, in order to prove this more general statement, one most likely will need to develop some new arguments.

We will note, however, that the maximal inequalities (and corresponding convergence results) in \cite{jx} and \cite{mmt} also do not need to make the assumptions about sections or the $T_j$'s commuting, though they still need each $T_j\in DS^+(\mathcal{M},\tau)$ (and \cite{mmt} also requires bounded weights to be used). 
\end{rem}

\subsection{$q$-Hartman sequences}

We now discuss another result that may be improved by Theorem \ref{t31}. For the rest of this paper we will restrict ourselves to the one parameter case $d=1$. We will also need to introduce some more notation.

Given a semifinite von Neumann algebra $(\mathcal{M},\tau)$, $1<r<\infty$, and $T\in DS^+(\mathcal{M},\tau)$, by the Jacobs-de Leeuw-Glicksberg decomposition (see \cite[Chapter 2.2]{kr}) of the space one may write
$$L_r(\mathcal{M},\tau)=\overline{\text{span}(\mathcal{U}_r(T))}\oplus\mathcal{V}_r(T),$$
where the closure is with respect to the norm of $L_r$ and
\begin{align*}
\mathcal{U}_r(T)=&\Big\{x\in L_r(\mathcal{M},\tau):T(x)=\lambda x\text{ for some }\lambda\in\mathbb{T}\Big\}, \\
\mathcal{V}_r(T)=&\Big\{x\in L_r(\mathcal{M},\tau):T^{n_j}(x)\to0\text{ weakly for some increasing }(n_j)_{j=0}^{\infty}\subseteq\mathbb{N}_0\Big\}.
\end{align*}
If $\mathcal{W}\subseteq W_1^{(1)}$, write
\begin{align*}
bWW_r^T(\mathcal{W}):=
\Big\{x\in  L_r(&\mathcal{M},\tau):\forall\epsilon>0\ \exists e\in\mathcal{P}(\mathcal{M})\text{ such that } 
\tau(e^\perp)\leq\epsilon \text{ and } \\
&(eM_n^\alpha(T)(x)e)_{n=1}^{\infty}\text{ converges in }\mathcal{M}\text{ for each }\alpha\in\mathcal{W}\Big\}.
\end{align*}

This latter space is the set of operators which satisfy a Wiener-Wintner type ergodic theorem on $L_r$ with respect to $T$ for weights in $\mathcal{W}$. In particular, when $x\in bWW_r^T(\mathcal{W})$ the averages $M_n^\alpha(T)(x)$ converge b.a.u. for every $\alpha\in\mathcal{W}$, with the projection being independent of $\alpha$ (though it may still depend on $\mathcal{W}$ itself).

In \cite{ob2} it was shown that if $T^n(x)\to0$ b.a.u. for every $x\in\mathcal{V}_r(T)$ for some $1<r<\infty$ and if $(T^n)_{n=0}^{\infty}$ is b.u.e.m. at zero on $(L_r,\|\cdot\|_r)$, then a Wiener-Wintner type result holds on $L_r(\mathcal{M},\tau)$ and with weights in $W_{1^+}^{(1)}\cap H^{(1)}$. However, for any other $p\neq r$ one could only guarantee the convergence on $L_p$ when the weights were in $W_\infty^{(1)}\cap H^{(1)}$ using the methods of that paper (unless the same assumptions also held on $L_p$ as well). 

These two assumptions on the iterates of $T$ were justified in \cite{ob2} through a large list of examples and easier to check conditions. For example, if the restriction of $T\in DS^+(\mathcal{M},\tau)$ to $L_2$ is self adjoint, or if $T$ is normal as a Hilbert space operator on $L_2(\mathcal{M},\tau)$ and $\sigma(T^n|_{L_2})\subset[0,1]$ for some $n\in\mathbb{N}$, then both assumptions hold for every $1<r<\infty$. In these cases Theorem \ref{t33} below doesn't say anything new as all sequences in $W_{1^+}^{(1)}\cap H^{(1)}$ can be used for each $L_p$, $1<p<\infty$. 

The operators considered in Theorems 2.7 and 4.6 of \cite{be} were only shown to satisfy these assumptions on $L_p(\mathcal{M},\tau)$ when $p=2$. Hence, the b.a.u. version of the Wiener-Wintner type ergodic theorem in \cite{ob2} allows weights in $W_{1^+}^{(1)}\cap H^{(1)}$ on the corresponding noncommutative $L_2$-space, while only weights in $W_\infty^{(1)}\cap H^{(1)}$ may be used on the other noncommutative $L_p$-spaces ($p\neq2$). The next result will improve this case to allow weights in $W_q^{(1)}\cap H^{(1)}$ for other $L_p$-spaces.

\begin{teo}\label{t33}
Let $(\mathcal{M},\tau)$ be a semifinite von Neumann algebra, $1<r<\infty$, and $T\in DS^+(\mathcal{M},\tau)$ be such that $(T^n)_{n=0}^{\infty}$ is b.u.e.m. at zero on $(L_r,\|\cdot\|_r)$ and $T^n(x)\to0$ b.a.u. for every $x\in\mathcal{V}_r(T)$. 
Then $L_r(\mathcal{M},\tau)=bWW_r^T(W_{1^+}^{(1)}\cap H^{(1)})$. Furthermore, if $1<p,q<\infty$ satisfies $\frac{1}{p}+\frac{1}{q}=1$, 
then $L_p(\mathcal{M},\tau)=bWW_p^T(W_q^{(1)}\cap H^{(1)})$.
\end{teo}
\begin{proof}
The claim regarding $L_r(\mathcal{M},\tau)=bWW_r^T(W_{1^+}^{(1)}\cap H^{(1)})$ is exactly \cite[Theorem 4.3]{ob2}. From this we deduce that $L_1\cap\mathcal{M}\subseteq bWW_r^T(W_{1^+}^{(1)}\cap H^{(1)})$, and looking at the definition of $bWW_r^T(W_{1^+}^{(1)}\cap H^{(1)})$ on sees that $L_1\cap\mathcal{M}\subseteq bWW_p^T(W_q^{(1)}\cap H^{(1)})$.

For $1<p,q<\infty$ with $\frac{1}{p}+\frac{1}{q}=1$, using Corollary \ref{c31} above in \cite[Theorem 3.1]{ob2} shows that that $bWW_p^T(W_q^{(1)}\cap H^{(1)})$ is closed in $L_p(\mathcal{M},\tau)$. Since $L_1\cap\mathcal{M}$ is both dense in $L_p$ and contained in $bWW_p^T(W_q^{(1)}\cap H^{(1)})$, it follows that $bWW_p^T(W_q^{(1)}\cap H^{(1)})=L_p(\mathcal{M},\tau)$.
\end{proof}

\begin{rem}\label{r32}
It was proven by Litvinov in \cite[Theorem 5.2]{li2} that  if $\mathcal{M}$ is a von Neumann algebra with a faithful normal tracial  state $\tau$, $\Phi:L_1\to L_1$ is a normal positive ergodic homomorphism (where $\Phi$ ergodic means $\Phi(x)=x$ with $x\in L_2$ implies $x=c\textbf{1}$ for some $c\in\mathbb{C}$) such that $\tau\circ \Phi=\tau$ and $\|\Phi(x)\|_\infty\leq\|x\|_\infty$ for every $x\in\mathcal{M}$, then $L_1(\mathcal{M},\tau)=bWW_1^{\Phi}(\mathcal{T}^{(1)})$.

Using an argument similar to Theorem \ref{t33} we can generalize the Wiener-Wintner ergodic theorem of \cite{li2} to allow $q$-Besicovitch sequences and state that $L_p(\mathcal{M},\tau)=bWW_p^{\Phi}(B_q^{(1)})$ for $\Phi$ as defined above when $1<p,q<\infty$ and $\frac{1}{p}+\frac{1}{q}=1$.
\end{rem}

\smallskip

\noindent\textbf{Acknowledgements.}
The author would like to thank Dr. L\'{e}onard Cadilhac for his input and suggestions that greatly improved the results of the paper. The author would also like to express his gratitude to Dr. Semyon Litvinov for his feedback in earlier versions of this paper.

\end{document}